\documentclass[reqno,centertags]{pspum-l}
\usepackage{amssymb,amsthm}
\usepackage[nobysame,abbrev]{amsrefs}
\newtheorem{thm}{Theorem}[section]

\newtheorem{lem}[thm]{Lemma}

\theoremstyle{definition}

\theoremstyle{remark}

\numberwithin{equation}{section}
\newcommand{\addsupport}[1]{} 

\newcommand{\bb}[1]{\mathbb #1}

\newcommand{\ol}{\overline}

\renewcommand{\Re}{\operatorname{Re}}
\renewcommand{\Im}{\operatorname{Im}}

\newcommand{\<}{\langle}
\renewcommand{\>}{\rangle}
\newcommand{\e}{{\rm e}}

\newcommand{\bU}{\breve U}
\newcommand{\bal}{\breve\alpha}
\newcommand{\bsi}{\breve\psi}
\newcommand{\bz}{\breve z}
\newcommand{\CCC}{\eta}

\newcommand{\bsm}{\left(\begin{smallmatrix}}
\newcommand{\esm}{\end{smallmatrix}\right)}

\begin{document}

\title{Stability for the inverse resonance problem for the CMV operator}

\author{Roman Shterenberg}
\address{Department of Mathematics, University of Alabama at Birmingham,
Birmingham, AL 35226-1170, USA}
\email{shterenb@math.uab.edu}

\author{Rudi Weikard}
\address{Department of Mathematics, University of Alabama at Birmingham,
Birmingham, AL 35226-1170, USA}
\email{rudi@math.uab.edu}

\author{Maxim Zinchenko}
\address{Department of Mathematics and Statistics, University of New Mexico, Albuquerque, NM 87131-0001, USA}
\email{maxim@math.unm.edu}

\thanks{The authors were supported in part by NSF grants DMS-0901015, DMS-0800906, and DMS-0965411, respectively.}

\dedicatory{Dedicated with great pleasure to Fritz Gesztesy on the occasion of his 60th birthday.}

\date{\today}


\begin{abstract}
For the class of unitary CMV operators with super-exponentially decaying Verblunsky coefficients we give a new proof of the inverse resonance problem of reconstructing the operator from its resonances - the zeros of the Jost function. We establish a stability result for the inverse resonance problem that shows continuous dependence of the operator coefficients on the location of the resonances.
\end{abstract}

\maketitle

\section{Introduction}
In recent years there has been a substantial interest in inverse resonance problems for operators that arise in mathematical physics. An inverse resonance problems asks whether an operator is determined by its resonances. One is typically interested in reconstructing coefficients of an operator from a given set of resonances and in stability of the solution of the inverse resonance problem under small perturbation of the resonances. The case of Schr\"odinger and Jacobi operators have received most of the attention, see for instance \cite{MR2164835, MR2519659, MR2851909, MR2860642, IK12, MR2104289, MNSW10, MR2594334, MR2348728} and the references therein.

The inverse resonance problem for a class of unitary operators known as CMV operators has been considered in \cite{MR2647154}. These operators are represented by five-diagonal infinite unitary matrices giving rise to unitary operators on $\ell^2(\bb N_0)$. They are closely connected with trigonometric moment problems and with orthogonal polynomials and finite measures on the unit circle. The entries in a CMV matrix are determined by the Verblunsky coefficients, a sequence of complex numbers in the unit disk. For background information on CMV matrices we refer the reader to Simon's two-volume monograph \cite{MR2105088,MR2105089}.

In this paper we continue our study of the inverse resonance problem for unitary CMV operators associated with sequences of exponentially decaying Verblunsky coefficients. For these operators we review the concept of a resonance and give a simplified proof of the uniqueness result in \cite{MR2647154} that the location of all resonances determines the Verblunsky coefficients uniquely. The main result of the present study is a stability result: Suppose two CMV operators in our class have resonances $z_n$ and $\bz_n$, respectively. Further suppose that these resonances are respectively close to each other as long as they are not too large. Then the Verblunsky coefficients are respectively close to each other, too. We will now make these statements precise. Throughout this paper $\gamma>1$, $\CCC>0$, and $Q>1$ denote fixed constants. We define the class $B_0(\gamma,\CCC,Q)$ as the set of those sequences of Verblunsky coefficients $\alpha:\bb N\to\bb D$ satisfying the following two conditions:
\begin{enumerate}
  \item $|\alpha_k|\leq \CCC\exp(-k^\gamma)$ for all $k\in\bb N$ and
  \item $\prod_{j=1}^\infty (1-|\alpha_j|)\geq 1/Q$.
\end{enumerate}
We note that the above two conditions are essentially independent as the first one enforces the decay of $|\alpha_k|$ for large $k$ and the second one bounds $|\alpha_k|$ away from $1$ for small $k$.
Also, two sets $\{z_1,...,z_N\}$ and $\{\bz_1,...,\bz_N\}$ of complex numbers will be called respectively $\varepsilon$-close if $|z_n-\bz_n|<\varepsilon$ for all $n\in\{1,..., N\}$.

The following technical result required for our proof of stability for the inverse resonance problem is of independent interest. We will prove it in Section \ref{S:distresonances}.
\begin{thm}\label{noresonance}
Suppose $\alpha$ is a sequence of Verblunsky coefficients in $B_0(\gamma,\CCC,Q)$ and $U$ is the associated CMV operator. Then there is a positive number $\delta$ such that $U$ has no resonances in the disk $\{z:|z|<1+\delta\}$.
\end{thm}

The main purpose of this paper is to prove the following theorem, which we will do in Section \ref{S:stability}.
\begin{thm}\label{main}
Suppose $\alpha$ and $\bal$ are two sequences of Verblunsky coefficients in $B_0(\gamma,\CCC,Q)$ and $U$ and $\bU$ are the associated CMV operators. Let $\delta$ be the number introduced in the previous theorem. Further suppose that, for two numbers $R>1$ and $\varepsilon\in(0,\delta/2)$, the resonances of $U$ and $\bU$ in the circle $|z|<R$, if there are any, are respectively $\varepsilon$-close. Then there is a constant $A_0$, depending only on $\gamma$, $\CCC$, and $Q$, such that
$$|\alpha_n-\bal_n|\leq A_0 (6Q^2)^{n} \bigg(\varepsilon+
\frac{(\log R)^{\gamma/(\gamma-1)}}{R}\bigg)$$
for all $n\in\bb N$.
\end{thm}

Theorem \ref{main} extends earlier results on stability of the inverse resonance problems for Schr\"odinger and Jacobi operators \cite{MNSW10, MR2594334, MR2348728} to the case of unitary CMV operators. In this note we present a new approach to the stability of the inverse resonance problems. Unlike the earlier work our approach does not rely on a heavy machinery of the transformation operators but instead uses the Schur algorithm - a simple recursion relation that arises naturally in the context of CMV operators. We point out that this approach is not specific to CMV operators only. There is also a similar simple recursive approach to the stability result of the inverse resonance problem for Jacobi operators.

The paper is organized as follows. In Section \ref{S:prelim} we introduce the basics of CMV operators, define the main objects, and state some known facts that are central to our study. In Section \ref{S:inverse} we set the stage for the stability result and give a new proof of the inverse resonance problem for CMV operators that first appeared in \cite{MR2647154}. Section \ref{S:stability} is devoted to the stability of the inverse resonance problem and contains the proof of our main Theorem \ref{main}.

Notation: In the following, we denote the set of all complex-valued sequences defined on $\bb N_0$ by $\bb C^{\bb N_0}$. The Hilbert space of all square summable complex-valued sequences is $\ell^2(\bb N_0)$ and its scalar product $\<\cdot,\cdot\>$ is linear in the second argument. Recall that the vectors $\delta_k\in\ell^2(\bb N_0)$, $k\in\bb N_0$, defined by the requirement that $\delta_k(n)$ equals Kronecker's $\delta_{k,n}$, form the standard basis in $\ell^2(\bb N_0)$. The open unit disk in the complex plane is denoted by $\bb D$.

\section{Preliminaries} \label{S:prelim}

\subsection{The CMV equations and the CMV operator}
The CMV equations are defined through a sequence of coefficients $\alpha:\bb N\to\bb D$; these are called Verblunsky coefficients. For $z\in\bb C\backslash\{0\}$ the CMV equations are the recursive equations
\begin{equation}\label{cmveq}
\binom uv(z,k)=T(z,k) \binom uv(z,k-1), \quad k\in\bb N
\end{equation}
where, using the abbreviation $\rho_k=\sqrt{1-|\alpha_k|^2}$,
$$T(z,k) = \begin{cases}
\frac{1}{\rho_{k}}
\begin{pmatrix}
\alpha_{k} & z \\ 1/z & \ol\alpha_{k}
\end{pmatrix},  & \text{$k$ odd},
\\
\frac{1}{\rho_{k}}
\begin{pmatrix}
\ol\alpha_{k} & 1 \\ 1 & \alpha_{k}
\end{pmatrix}, & \text{$k$ even}.
\end{cases}$$
It is clear that the space of solutions of these equations is $2$-dimensional. A basis of solutions is given by the sequences $\vartheta(z,\cdot)$ and $\varphi(z,\cdot)$ defined by the initial conditions
\begin{align} \label{IC}
\vartheta(z,0)=(-1,1)^\top \text{ and } \varphi(z,0)=(1,1)^\top.
\end{align}
Clearly $\vartheta(\cdot,k)$ and $\varphi(\cdot,k)$ are analytic in $\bb C\backslash\{0\}$ for any $k\in\bb N_0$.
Since
$$\ol{T(1/\ol z,k)} =
\begin{pmatrix}0 & 1 \\ 1 & 0 \end{pmatrix}
T(z,k)
\begin{pmatrix}0 & 1 \\ 1 & 0 \end{pmatrix}$$
$({\ol v},{\ol u})^\top(1/\ol z,\cdot)$ satisfies the CMV equations if $(u,v)^\top(z,\cdot)$ does. This implies, taking the initial conditions into account, that
\begin{equation} \label{varrefl}
\varphi(z,k)=\begin{pmatrix}0 & 1 \\ 1 & 0 \end{pmatrix}\ol{\varphi(1/\ol z,k)} \text{ and }
\vartheta(z,k)=-\begin{pmatrix}0 & 1 \\ 1 & 0 \end{pmatrix}\ol{\vartheta(1/\ol z,k)}
\end{equation}
whenever $z\in\bb C\backslash\{0\}$.

To define the CMV operator set first
$$\Theta_k = \begin{pmatrix} -\alpha_k & \rho_k \\ \rho_k & \ol\alpha_k\end{pmatrix}.$$
These blocks are then used to define $W=\bigoplus_{k=1}^\infty \Theta_{2k-1}$ and $V=1\oplus\left(\bigoplus_{k=1}^\infty \Theta_{2k}\right)$ where $1$ is interpreted as a $1\times1$ block. Finally the product $VW$ is denoted by $U$ and is called the CMV operator. $U$, $V$, and $W$ are defined on $\bb C^{\bb N_0}$ (and map to that space). Their restrictions to $\ell^2(\bb N_0)$ are unitary operators which we denote using the same letters as the precise meaning will always be clear from the context.

The following lemma was established in \cite{MR2220038}:
\begin{lem}
Suppose $z\in\bb C\backslash\{0\}$.
$(u,v)^\top(z,\cdot)$ is a solution of the CMV equations \eqref{cmveq} if and only if
$$\begin{pmatrix}U&0\\ 0&V\end{pmatrix} \binom uv(z,\cdot)=[u(z,\cdot)+(v(z,0)-u(z,0))\delta_0] \binom z1.$$
\end{lem}

Hence if $u(z,\cdot)$ is the first component of $\varphi(z,\cdot)$ we have $Uu=zu$.

Now suppose that $|z|\neq1$ so that $z$ is in the resolvent set of the unitary operator $U$ and set $u(z,\cdot)=2z(U-z)^{-1}\delta_0$ which is in $\ell^2(\bb N_0)$. Define
\begin{equation}\label{mFunction}
m(z)=1+u(z,0)=\<\delta_0,(U+z)(U-z)^{-1}\delta_0\>
\end{equation}
which is analytic in $\bb D$, and, assuming also $z\neq0$,
$$\omega(z,\cdot)=\vartheta(z,\cdot)+m(z)\varphi(z,\cdot).$$
Both components of $\omega(z,\cdot)$ are square summable (the first component is $u$ and the second is $V^{-1}(u+2\delta_0)$). Moreover, $\omega(z,\cdot)$ and its constant multiples are the only square summable solutions of the CMV equations \eqref{cmveq} since the unitary operator $U$ can not have eigenvalues away from the unit circle. Thus, employing \eqref{varrefl}, we find that
$$-\begin{pmatrix}0 & 1 \\ 1 & 0 \end{pmatrix}\ol{\omega(1/\ol z,\cdot)}
 =\vartheta(z,\cdot)-\ol{m(1/\ol z)} \varphi(z,\cdot)$$
is equal to $\omega(z,\cdot)$ which implies that $m(z)=-\ol{m(1/\ol z)}$.

The function $m$ is called the Weyl-Titchmarsh $m$-function while the sequence $\omega(z,\cdot)$ is called the Weyl-Titchmarsh solution of the CMV equations \eqref{cmveq}. We also note that, for $z=0$ we get $u(0,\cdot)=0$, $m(0)=1$, and $v(0,\cdot)=2 \delta_0$ so that, for every $k\in\bb N_0$, the singularity of $\omega(\cdot,k)$ at $0$ is removable.

It follows from \eqref{mFunction} via the spectral theorem that
\begin{equation*}
m(z) = \oint_{\partial \bb D}\frac{\zeta+z}{\zeta-z}d\mu(\zeta)
 =\frac1{2\pi} \int_{-\pi}^\pi \frac{\e^{it}+z}{\e^{it}-z} d\tilde\mu(t),
\end{equation*}
where $d\mu$ denotes the spectral measure associated with the unitary operator $U$ and the cyclic vector $\delta_0$ and $\tilde\mu(t)=2\pi\mu(e^{it})$ gives rise to the corresponding measure on $[-\pi,\pi]$. The case $z=0$ shows that $d\mu$ is a probability measure. Since $(\e^{it}+z)/(\e^{it}-z)$ has positive real part for all $z\in\bb D$, it follows that $m$ is a Caratheodory function, that is, $m$ is analytic on $\bb D$, $m(0)=1$, and $\Re m(z)>0$ for $|z|<1$.

Employing the Neumann series for $(U-z)^{-1}$ in \eqref{mFunction} shows
$$m(z)=1+2\sum_{n=1}^\infty z^{n} \<\delta_0, U^{-n}\delta_0\>.$$
This implies $m^{(n)}(0)=2n! \<\delta_0, U^{-n}\delta_0\>$ so that
\begin{equation}\label{alpha}
m'(0)=-2\bar\alpha_1 \text{ and } m''(0)=4\bar\alpha_1^2-4\rho_1^2\bar\alpha_2.
\end{equation}

\subsection{Jost solutions, the Jost function, and resonances}
Assuming super-exponential decay of the Verblunsky coefficients, i.e.,
$$|\alpha_k|\leq \CCC\exp(-k^\gamma)$$
Jost solutions of the CMV equations were defined and investigated in \cite{MR2647154}. We repeat here briefly the most important results.

Defining
$$\zeta_k=\begin{cases}z, & \text{$k$ odd} \\ 1, & \text{$k$ even}\end{cases}$$
it was proved in \cite{MR2647154} that the Volterra-type equations
\begin{align}
F(z,k)=\binom{1}{0} - \sum_{n=k+1}^\infty
\begin{pmatrix} 0&\alpha_n\zeta_n\\ z^{n-k-1}\ol\alpha_n\zeta_{k+1}\end{pmatrix}F(z,n),
 \quad k\in\bb N_0, \label{3.7}
\end{align}
have a unique solution for any complex number $z$. Either component of $F(\cdot,k)$ is an entire function of growth order zero and, if $|z|\geq1$,
\begin{equation}\label{growthorder}
\|F(z,0)\|\leq \exp(\CCC+2K(z)\log|z|) \prod_{n=1}^\infty (1+|\alpha_n|)
\end{equation}
where $\|\cdot\|$ denotes the 2-norm in $\bb C^2$ and
$$K(z)=\lfloor(\log 2|z|^2)^{\frac{1}{\gamma-1}}\rfloor.$$
We also recall that
\begin{equation}\label{asyF}
\|F(z,k)-\binom10\|\leq \beta(z,k)\exp(\beta(z,k))
\end{equation}
where $\beta(z,k)=\sum_{n=k+1}^\infty |\alpha_n|\max\{1,|z|^{2n-1}\}$.

Setting $C_k=\prod_{j=k+1}^\infty \rho_j^{-1}$, it is straightforward to show that the sequence $\nu(z,\cdot)$ defined by
\begin{align} \label{nuF}
\nu(z,k)=2 z^{\lceil k/2 \rceil} C_k \begin{pmatrix}0&1\\1&0\end{pmatrix}^{k+1}F(z,k)
\end{align}
satisfies the CMV equations \eqref{cmveq} as does the sequence
$$\tilde \nu(z,k)=\begin{pmatrix}0&1\\1&0\end{pmatrix} \ol{\nu(1/\ol z,k)}.$$
If $|z|<1$ both components of $\nu(z,\cdot)$ are in $\ell^2(\bb N_0)$ so that $\nu(z,\cdot)$ must be a multiple of the Weyl-Titchmarsh solution $\omega(z,\cdot)$, i.e.,
$$\nu(z,\cdot)=\psi_0(z)\omega(z,\cdot)=\psi_0(z)(\vartheta(z,\cdot)+m(z)\varphi(z,\cdot)), \quad |z|<1$$
for some appropriate function $\psi_0$. Evaluating at $k=0$ using \eqref{IC} and \eqref{nuF} yields
\begin{align}
&\psi_0(z) = \frac{(-1,1)\nu(z,0)}{2} = C_0(1,-1)F(z,0), \label{psiF}\\
&\psi_0(z)m(z) = \frac{(1,1)\nu(z,0)}{2} = C_0(1,1)F(z,0),
\end{align}
hence $\psi_0$ and $\psi_0 m$ extend to entire functions. Consequently $m$ extends to a meromorphic function on $\bb C$ which we will denote by $M$ (we emphasize that $M(z)\neq m(z)$ for $|z|>1$). Examining also $\tilde\nu$ we obtain the relationships
$$\nu(z,\cdot)=\psi_0(z)(\vartheta(z,\cdot)+M(z)\varphi(z,\cdot))$$
and
$$\tilde\nu(z,\cdot)=\ol{\psi_0(1/\ol z)}(-\vartheta(z,\cdot)+\ol{M(1/\ol z)}\varphi(z,\cdot))$$
which are valid for any $z\in\bb C\backslash\{0\}$. The solutions $\nu(z,\cdot)$ and $\tilde\nu(z,\cdot)$ are called respectively the Jost solutions of the CMV equations if $|z|<1$ or $|z|>1$. The function $\psi_0$ is called the Jost function. It is an entire function of growth order zero. Its zeros are called resonances.

We end this section with the following observation. Since $\det T(z,k)=-1$ we find
$$\det(\nu(z,k),\tilde\nu(z,k))=(-1)^k \det(\nu(z,0),\tilde\nu(z,0)).$$
The asymptotic behavior of $\nu(z,k)$ and $\tilde\nu(z,k)$ as $k$ tends to infinity shows that the left hand side is equal to $4(-1)^{k+1}$. If we now pick $z$ on the unit circle so that $z=1/\ol z$ we obtain from this
\begin{equation}\label{psiReM}
1=|\psi_0(z)|^2\Re(M(z)).
\end{equation}

\subsection{The Schur algorithm}
Closely related to the concept of a Caratheodory function is that of a Schur function, that is, a function defined and analytic on $\bb D$ whose modulus never exceeds one. Indeed if $f$ is a Caratheodory function then $(f-1)/(f+1)$ is a Schur function, while $(1+g)/(1-g)$ is a Caratheodory function if $g$ is a Schur function. Note also that, by Schwarz's lemma,  $z\mapsto g(z)/z$ is a Schur functions if $g$ is a Schur function and $g(z)=0$.

For $k\in\bb N_0$ and $z\in\bb D$ we define now the functions
\begin{align} \label{nuPhi}
\Phi_{2k}(z)=\frac1z \frac{(1,0)\omega(z,2k)}{(0,1)\omega(z,2k)}
\text{ and }
\Phi_{2k+1}(z)=\frac{(0,1)\omega(z,2k+1)}{(1,0)\omega(z,2k+1)}
\end{align}
and we note that in place of $\omega$ we may as well put $\nu$ since these are multiples of each other.

The initial conditions satisfied by the Weyl-Titchmarsh solution $\omega$ show that
\begin{equation}\label{Schur}
\Phi_0(z)=\frac1z \frac{m(z)-1}{m(z)+1}
\end{equation}
which is a Schur function.

Using the CMV equations \eqref{cmveq} one may check that
$$\Phi_k(z)=\frac1z S(\alpha_k,\Phi_{k-1}(z))$$
where $S(w,\cdot)$ is the M\"obius transform
$$z\mapsto S(w,z)=\frac{z+\ol w}{1+wz}$$
which maps $\bb D$ to $\bb D$ provided $w\in\bb D$.

Next, taking \eqref{alpha} into account, one sees that $\Phi_0(0)=\frac12 m'(0)=-\bar\alpha_1$. This shows that $\Phi_1$ is a Schur function and we find, again by \eqref{alpha}, that $\Phi_1(0)=\Phi_0'(0)/(1-|\alpha_1|^2)=-\bar\alpha_2$.

Consider now the truncated sequence of Verblunsky coefficients $n\mapsto\alpha_{2N+n}$. The Jost solution for this problem is given by $k\mapsto z^{-N} \nu(z,2N+k)$. Consequently, the function $\Phi_{2N}$ plays the same role for the truncated sequence as $\Phi_0$ plays for the full sequence. Therefore we have $\Phi_{2N}(0)=-\bar\alpha_{2N+1}$ and $\Phi_{2N+1}(0)=-\bar\alpha_{2N+2}$. Thus any of the functions $\Phi_k$ is a Schur function and
\begin{equation}\label{verfromschur}
\Phi_k(0)=-\bar\alpha_{k+1}.
\end{equation}

The hyperbolic distance on $\bb D$, given by
$$d[w_1,w_2] = 2\tanh^{-1}\left|\frac{w_1-w_2}{1-\ol w_1 w_2}\right|
 =\log\frac{1+|w_1-w_2|/|1-\ol w_1 w_2|}{1-|w_1-w_2|/|1-\ol w_1 w_2|},$$
is invariant under M\"obius transforms which map $\bb D$ onto itself. Hence, employing the triangle inequality and the fact that, for $|z|<1$, we have $d[z\Phi_k(z),0]\leq d[\Phi_k(z),0]$,
$$d[\Phi_{k-1}(z),0]=d[z\Phi_k(z),\ol\alpha_k]\leq d[\Phi_{k}(z),0]+d[0,\ol\alpha_k].$$
Inequality \eqref{asyF} combined with \eqref{nuF} and \eqref{nuPhi} implies that $\Phi_k(z)$ tends uniformly to zero as $k$ tends to infinity. Hence we may sum up the telescoping series resulting from the previous inequality to get
\begin{equation*}
d[\Phi_{0}(z),0]\leq\sum_{k=1}^\infty d[0,\ol\alpha_k]
 = \sum_{k=1}^\infty \log\frac{1+|\alpha_k|}{1-|\alpha_k|}\leq \log Q^2.
\end{equation*}
This, in turn, gives us the estimate
\begin{equation}\label{Phiest}
\frac{1+|\Phi_0(z)|}{1-|\Phi_0(z)|}\leq \prod_{k=1}^\infty \frac{1+|\alpha_k|}{1-|\alpha_k|}
\leq Q^2
\end{equation}
which holds for all $z\in\bb D$ and hence also in the closure of $\bb D$.

\section{The inverse problem} \label{S:inverse}
The main purpose of \cite{MR2647154} was to prove the following theorem.
\begin{thm}\label{T3.1}
The locations (and multiplicities) of the zeros of the Jost function, that is, the resonances, associated with the CMV equations \eqref{cmveq} determine u\-nique\-ly the Verblunsky coefficients, provided these satisfy $|\alpha_n|\leq \CCC\exp(-n^\gamma)$.
\end{thm}
The strategy in \cite{MR2647154} was to show that the zeros of $\psi_0$ determine the Weyl-Titchmarsh $m$-function. The $m$-function, in turn, determines the Verblunsky coefficients as was shown in \cite{MR2220038}. In the following we prove this fact directly, because the proof of our stability result relies on the details of the relationship between the resonances and the Verblunsky coefficients.

\begin{proof}[Proof of Theorem \ref{T3.1}]
As we know from equation \eqref{verfromschur}, the Verblunsky coefficients are determined by the Schur functions $\Phi_k$. Of these the first one is determined by $m$ while the subsequent ones are found recursively via the M\"obius transform $S$.

To find $m$ we call on Schwarz's integral formula which says that
\begin{equation}\label{schwarz}
m(z)=\frac1{2\pi}\int_{-\pi}^\pi \frac{r\e^{it}+z}{r\e^{it}-z}\Re(m(r\e^{it})) dt
\end{equation}
as long as $|z|<r<1$. According to the first inequality in \eqref{Phiest} $m=(1+z\Phi_0(z))/(1-z\Phi_0(z))$ is uniformly bounded in $\bb D$  and hence has no poles in its closure. Therefore we may take the limit $r\to1$ in \eqref{schwarz} under the integral. This fact and equation \eqref{psiReM} show that
\begin{equation*}
d\tilde\mu(t)=\Re(M(\e^{it}))dt=\frac{dt}{|\psi_0(\e^{it})|^2}
\end{equation*}
so that
\begin{equation}\label{mfrompsi}
m(z)=\frac1{2\pi}\int_{-\pi}^\pi \frac{\e^{it}+z}{\e^{it}-z} \frac{dt}{|\psi_0(\e^{it})|^2}
\end{equation}
for $|z|<1$.

Finally we have to show that $\psi_0$ is determined from its zeros (the resonances). Since it is an entire function of growth order zero Hadamard's factorization theorem gives
$\psi_0(z) =\psi_0(0) \Pi(z)$ where
$$\Pi(z)=\prod_{n=1}^\infty (1-z/z_n)$$
and where the $z_n$ are the zeros of $\psi_0$ repeated according to their multiplicities. We claim that the value $\psi_0(0)$ is also determined by the $z_n$. Indeed, evaluating \eqref{mfrompsi} at $z=0$, using that $m(0)=1$, gives
\begin{equation}\label{psi00}
|\psi_0(0)|^2=\frac1{2\pi} \int_{-\pi}^\pi \frac{dt}{|\Pi(\e^{it})|^2}.
\end{equation}
This will prove our claim if we can show that $\psi_0(0)$ is positive. To this end we note that $F(0,k)=(1,0)^\top$ is the unique solution of the Volterra-type equations \eqref{3.7} for $z=0$. Thus
\begin{equation}\label{C0}
\psi_0(0)=C_0=\prod_{j=1}^\infty \rho_j^{-1}\geq1.
\end{equation}
\end{proof}

\section{Stability} \label{S:stability}
In this section we prove Theorem \ref{main}. Throughout the section we assume that any sequence of Verblunsky coefficients is from the class $B_0(\gamma,\CCC,Q)$. Subsequently we will be using repeatedly the following elementary facts:
\begin{enumerate}
  \item If $|x|\leq 1/2$, then $|\log(1-x)|\leq2|x|$.
  \item $|\e^u-1|\leq |u| \e^{|u|}$ for all $u\in\bb C$. Moreover, if $|u|\leq 1/2$, then $|\e^u-1|\leq 2|u|$.
\end{enumerate}

\subsection{Upper and lower bound for $\Pi$}
Upper bounds on $\Pi$ follow from \eqref{growthorder} since $\Pi(z)=(1,-1)F(z,0)$ by \eqref{psiF} and \eqref{C0}. Indeed, using that $\prod_{j=1}^\infty (1+|\alpha_j|)\leq Q$ for $\alpha\in B_0(\gamma,\CCC,Q)$, we find
\begin{equation}\label{logPi}
\log|\Pi(z)|\leq \log(\sqrt2Q)+\CCC+(\log2|z|^2)^{\gamma/(\gamma-1)}
\end{equation}
as long as $|z|\geq1$.

We need lower bounds on $\Pi$ only on the unit circle. Recall that by \eqref{psiReM} we have $\Re(M(z))=|\psi_0(z)|^{-2}=C_0^{-2}|\Pi(z)|^{-2}$ if $|z|=1$. Note also that
$$\Re(M(z))=\frac{1-|z\Phi_0(z)|^2}{|1-z\Phi_0(z)|^2}\leq \frac{1+|\Phi_0(z)|}{1-|\Phi_0(z)|}.$$
Combining these facts with \eqref{C0} and \eqref{Phiest} gives
\begin{equation}\label{lowerbound}
|\Pi(z)|^{-2}\leq C_0^2 \prod_{n=1}^\infty \frac{1+|\alpha_n|}{1-|\alpha_n|}
=\prod_{n=1}^\infty \frac{1}{(1-|\alpha_n|)^2}\leq Q^2.
\end{equation}

\subsection{Distribution of resonances} \label{S:distresonances}
Now let $N(r)$ denote the number of zeros of $\Pi$ in the open disk of radius $r$ centered at zero. We know that $N(r)=0$ for $r\leq 1$. To deal with $r\geq1$ we use Jensen's formula, the estimate \eqref{logPi}, and the inequality $2(a+b)^p \leq (2a)^p+(2b)^p$, which holds for $a,b\geq0$ and $p\geq1$, to find
$$N(r)\leq \int_r^{er} \frac{N(t)}t dt\leq \int_0^{er} \frac{N(t)}t dt
 =\frac1{2\pi} \int_0^{2\pi} \log|\Pi(\e r\e^{it})|dt
 \leq A_1+\frac{(\log r^4)^p}2$$
where $p=\gamma/(\gamma-1)$ and $A_1$ is a suitable constant which depends only on $Q$, $\CCC$, and $\gamma$.
From this we obtain
$$\sum_{|z_n|\geq R}\frac{1}{|z_n|}
 =\int_R^\infty \frac{dN(t)}{t}=-\frac{N(R)}R +\int_R^\infty \frac{N(t)}{t^2}dt
 \leq \frac{A_1}{R}+\frac{4^p}{2}\Gamma(p+1,\log R)$$
where $\Gamma$ denotes the incomplete Gamma function \cite[Sect.\ 6.5]{MR0167642}. In particular, we get
$$\sum_{n=1}^\infty\frac1{|z_n|}\leq A_1+\frac{4^p}{2}\Gamma(p+1)$$
by setting $R=1$. The asymptotic behavior of the incomplete Gamma function \cite[Eq.\ 6.5.32]{MR0167642} shows now that
\begin{equation} \label{zninv}
\sum_{|z_n|\geq R}\frac{1}{|z_n|} \leq A_2 \frac{(\log R)^p}R
\end{equation}
if $R\geq 2$ and $A_2$ is a suitable constant (depending on $Q$, $\CCC$, and $\gamma$).

Next we estimate how close resonances can be to the unit circle. From \eqref{logPi} (making use of the maximum principle) we know that there is a constant $L$ depending only on $Q$, $\CCC$, and $\gamma$ such that $|\Pi(z)|\leq L$ whenever $|z|\leq \e$.
Cauchy's estimate gives $|\Pi'(a)|\leq L/(\e-|a|)$ for any point $a$ with $|a|<\e$. Let $1+\delta=(QL+\e)/(QL+1)$ and $z_0$ a point on the unit circle. Then
$$|\Pi(t z_0)|\geq |\Pi(z_0)|-\int_1^t |\Pi'(z_0s)|ds\geq \frac1Q-(t-1)\frac{L}{\e-t}.$$
Since this is positive as long as $1\leq t<1+\delta$ we have established that there are no zeros of $\psi_0$, i.e., no resonances, in the disk $|z|<1+\delta$ for any operator $U$ from the class $B_0(\gamma,\CCC,Q)$.

\subsection{Comparing $\Pi$ and $\breve\Pi$}
We assume now that we have two CMV operators $U$ and $\bU$ with Verblunsky coefficients $\alpha_n$ and $\bal_n$, respectively. More generally, any quantity associated with $\bU$ will have a $\breve{}$ accent to distinguish it from the corresponding quantity associated with $U$. Both $\alpha$ and $\bal$ are in $B_0(\gamma,\CCC,Q)$.

Our goal is to show that the differences $|\alpha_n-\bal_n|$ are arbitrarily small provided that the resonances of the associated operators $U$ and $\bU$ in a sufficiently large disk (of radius $R$) are respectively $\varepsilon$-close for a sufficiently small $\varepsilon$ , i.e., to prove Theorem \ref{main}.

We will henceforth always assume $R\geq2$ and $\varepsilon\leq\delta/2$. We begin by looking at those factors in $\Pi$ and $\breve\Pi$ associated with resonances which are respectively close to each other, i.e., the resonances in a disk of radius $R$. Let $N$ be their number and assume $|z|=1$. Thus $|\bz_n-z_n|\leq\varepsilon$ and $|\bz_n-z|,|z_n-z|\geq\delta>0$ for $1\leq n\leq N$. Since $\varepsilon/\delta\leq 1/2$ we have the following estimate
\begin{align}
  \left|\log\left(\prod_{n=1}^N \frac{1-z/\bz_n}{1-z/z_n}\right)\right|
  \leq \sum_{n=1}^N \left|\log\left(1-z\frac{z_n-\bz_n}{(z_n-z)\bz_n}\right)\right|
  \leq \frac{2\varepsilon}{\delta} \sum_{n=1}^\infty |\bz_n|^{-1}. \label{smallresonances}
\end{align}
We showed above that the sum on the right is bounded by $A_1+\Gamma(p+1)/2$.

Next we turn to the terms associated with large resonances and show that these are negligible. Indeed, we get for $|z_n|, |\bz_n|\geq R\geq 2$ and $|z|=1$
$$\left|\log\prod_{n=N+1}^\infty \frac{1-z/\bz_n}{1-z/z_n}\right| \leq 2\sum_{n=N+1}^\infty\left(\frac{1}{|\bz_n|}+\frac1{|z_n|}\right)$$
so that, with the aid of \eqref{zninv}, we arrive at the estimate
$$\left|\log\prod_{n=N+1}^\infty \frac{1-z/\bz_n}{1-z/z_n}\right|
\leq 4A_2 \frac{(\log R)^p}R.$$

Combining this estimate with \eqref{smallresonances} and denoting by $A_3$
a suitable constant depending only on $Q$, $\CCC$, and $\gamma$, we obtain
\begin{equation}\label{qPi}
\left|\frac{\breve\Pi(z)}{\Pi(z)}-1\right|\leq A_3 \left(\varepsilon+\frac{(\log R)^p}R\right)
\end{equation}
provided that $|z|\leq1$.

\subsection{Comparing $\psi_0$ and $\bsi_0$}
Since, by \eqref{qPi} and \eqref{lowerbound},
$$\left|\frac1{|\Pi(\e^{it})|^2}-\frac1{|\breve\Pi(\e^{it})|^2}\right|\leq 2Q^2 A_3\bigg(\varepsilon+\frac{(\log R)^p}R\bigg)$$
and since $C_0,\breve C_0\geq1$ we get from \eqref{psi00} that
\begin{equation*}
|C_0^{-2}-\breve C_0^{-2}|\leq|C_0^{2}-\breve C_0^{2}|
 \leq 2Q^2 A_3(\varepsilon+(\log R)^p/R).
\end{equation*}
Thus, whenever $|z|\leq1$,
\begin{equation}\label{diffpsipsi}
\left|\frac1{|\psi_0(z)|^2}-\frac1{|\bsi_0(z)|^2}\right|
\leq A_4 \bigg(\varepsilon+\frac{(\log R)^p}R\bigg)
\end{equation}
where $A_4$ depends only on $Q$, $\CCC$, and $\gamma$.

\subsection{Comparing Verblunsky coefficients}
Suppose $1-|w|, 1-|\breve w| \geq 1/Q$ and $z,\breve z$ are in the closed unit disk. Then
$$|S(w,z)-S(\breve w, \breve z)| \leq Q^2 (4|w-\breve w|+2|z-\breve z|).$$
Since by assumption $1-|\alpha_n|\geq\prod_{j=1}^\infty (1-|\alpha_j|)\geq 1/Q$,
it follows that for all $z$ on the unit circle,
\begin{align} \label{PointwiseEstimate}
|\Phi_k(z)-\breve\Phi_k(z)| \leq Q^2 (4|\alpha_k-\bal_k|+2|\Phi_{k-1}(z)-\breve\Phi_{k-1}(z)|).
\end{align}

Let $\|\cdot\|_p$ denote the $L^p$-norm on the unit circle with respect to the normalized Lebesgue measure. By Gauss's mean value theorem
$$
|\alpha_k-\bal_k| \leq \|\Phi_{k-1}-\breve\Phi_{k-1}\|_1, \quad k\in\bb N,
$$
and hence \eqref{PointwiseEstimate} yields,
\begin{align*}
\|\Phi_k-\breve\Phi_k\|_1 \leq 6Q^2 \|\Phi_{k-1}-\breve\Phi_{k-1}\|_1, \quad k\in\bb N.
\end{align*}
Thus we have
\begin{align*}
|\alpha_k-\bal_k| \leq (6Q^2)^{k-1} \|\Phi_0-\breve\Phi_0\|_1, \quad k\in\bb N,
\end{align*}
by induction. Since for all $z$ on the unit circle $\Re M(z) = 1/|\psi_0(z)|^2$ and  $\Re \breve M(z) =1/|\breve\psi_0(z)|^2$ are nonnegative, it follows from \eqref{Schur} that $\|\Phi_0-\breve\Phi_0\|_1 \leq 2\|M-\breve M\|_1 \leq 2\|M-\breve M\|_2$. The imaginary parts of $M$ and $\breve M$ can be obtained from the Hilbert transform of the respective real parts.
Since the Hilbert transform is unitary on the space of square integrable functions we have $\|\Im M-\Im \breve M\|_2 = \|\Re M-\Re \breve M\|_2$ and hence $\|M-\breve M\|_2 \leq \sqrt2\|\Re M-\Re\breve M\|_2$. Thus, we get from \eqref{diffpsipsi}
\begin{multline}
|\alpha_k-\bal_k| \leq 2\sqrt2(6Q^2)^{k-1} \bigg\|\frac{1}{|\psi_0(z)|^2} - \frac{1}{|\breve\psi_0(z)|^2}\bigg\|_2 \\
\leq 2\sqrt2(6Q^2)^{k-1} A_4\bigg(\varepsilon+\frac{(\log R)^p}R\bigg). \label{eq4.8}
\end{multline}

Setting $A_0=2\sqrt2 A_4/(6Q^2)$ completes the proof of Theorem \ref{main}.

Estimate \eqref{eq4.8} becomes worse with increasing $k$. Eventually, of course we will have $|\alpha_k-\bal_k|\leq 2\CCC\e^{-k^\gamma}$ just by using our hypothesis on super-exponential decay of the Verblunsky coefficients. Using the worst possible case and introducing yet another approriate constant $A_5$ gives us the uniform estimate
$$|\alpha_k-\bal_k|\leq A_5 \left(\varepsilon +\frac{(\log R)^p}R\right)^{1/\log(6\e Q^2)}, \quad k\in\bb N.$$


\begin{bibdiv}
\begin{biblist}

\bib{MR0167642}{book}{
   author={Abramowitz, Milton},
   author={Stegun, Irene A.},
   title={Handbook of mathematical functions with formulas, graphs, and
   mathematical tables},
   series={National Bureau of Standards Applied Mathematics Series},
   volume={55},
   publisher={For sale by the Superintendent of Documents, U.S. Government
   Printing Office, Washington, D.C.},
   date={1964},
   pages={xiv+1046},
   review={\MR{0167642 (29 \#4914)}},
}

\bib{MR2164835}{article}{
      author={Brown, B.~Malcolm},
      author={Naboko, Sergey},
      author={Weikard, Rudi},
       title={The inverse resonance problem for {J}acobi operators},
        date={2005},
        ISSN={0024-6093},
     journal={Bull. London Math. Soc.},
      volume={37},
      number={5},
       pages={727\ndash 737},
         url={http://dx.doi.org/10.1112/S0024609305004674},
      review={\MR{2164835 (2006e:39032)}},
}

\bib{MR2519659}{article}{
      author={Brown, B.~Malcolm},
      author={Naboko, Serguei},
      author={Weikard, Rudi},
       title={The inverse resonance problem for {H}ermite operators},
        date={2009},
        ISSN={0176-4276},
     journal={Constr. Approx.},
      volume={30},
      number={2},
       pages={155\ndash 174},
         url={http://dx.doi.org/10.1007/s00365-008-9037-8},
      review={\MR{2519659 (2011b:47065)}},
}

\bib{MR2220038}{article}{
      author={Gesztesy, Fritz},
      author={Zinchenko, Maxim},
       title={Weyl-{T}itchmarsh theory for {CMV} operators associated with
  orthogonal polynomials on the unit circle},
        date={2006},
        ISSN={0021-9045},
     journal={J. Approx. Theory},
      volume={139},
      number={1-2},
       pages={172\ndash 213},
         url={http://dx.doi.org/10.1016/j.jat.2005.08.002},
      review={\MR{2220038 (2007f:47027)}},
}

\bib{MR2851909}{article}{
      author={Iantchenko, Alexei},
      author={Korotyaev, Evgeny},
       title={Periodic {J}acobi operator with finitely supported perturbation
  on the half-lattice},
        date={2011},
        ISSN={0266-5611},
     journal={Inverse Problems},
      volume={27},
      number={11},
       pages={115003, 26},
         url={http://dx.doi.org/10.1088/0266-5611/27/11/115003},
      review={\MR{2851909}},
}

\bib{MR2860642}{article}{
      author={Iantchenko, Alexei},
      author={Korotyaev, Evgeny},
       title={Periodic {J}acobi operator with finitely supported perturbations:
  the inverse resonance problem},
        date={2012},
        ISSN={0022-0396},
     journal={J. Differential Equations},
      volume={252},
      number={3},
       pages={2823\ndash 2844},
         url={http://dx.doi.org/10.1016/j.jde.2011.09.034},
      review={\MR{2860642}},
}

\bib{IK12}{article}{
      author={Iantchenko, Alexei},
      author={Korotyaev, Evgeny},
       title={Resonances for periodic {J}acobi operators with finitely
  supported perturbations},
        date={2012},
        ISSN={0022-247X},
     journal={J. Math. Anal. Appl.},
      volume={388},
      number={2},
       pages={1239–\ndash 1253},
  url={http://www.sciencedirect.com/science/article/pii/S0022247X11010365},
}

\bib{MR2104289}{article}{
      author={Korotyaev, Evgeni},
       title={Stability for inverse resonance problem},
        date={2004},
        ISSN={1073-7928},
     journal={Int. Math. Res. Not.},
      number={73},
       pages={3927\ndash 3936},
         url={http://dx.doi.org/10.1155/S1073792804140609},
      review={\MR{2104289 (2005i:81186)}},
}

\bib{MNSW10}{article}{
      author={Marletta, Marco},
      author={Naboko, Sergey},
      author={Shterenberg, Roman},
      author={Weikard, Rudi},
       title={On the inverse resonance problem for {J}acobi operators --
  uniqueness and stability},
     journal={To appear in J. Anal. Math.},
}

\bib{MR2594334}{article}{
      author={Marletta, Marco},
      author={Shterenberg, Roman},
      author={Weikard, Rudi},
       title={On the inverse resonance problem for {S}chr\"odinger operators},
        date={2010},
        ISSN={0010-3616},
     journal={Comm. Math. Phys.},
      volume={295},
      number={2},
       pages={465\ndash 484},
         url={http://dx.doi.org/10.1007/s00220-009-0928-8},
      review={\MR{2594334 (2011a:34203)}},
}

\bib{MR2348728}{article}{
      author={Marletta, Marco},
      author={Weikard, Rudi},
       title={Stability for the inverse resonance problem for a {J}acobi
  operator with complex potential},
        date={2007},
        ISSN={0266-5611},
     journal={Inverse Problems},
      volume={23},
      number={4},
       pages={1677\ndash 1688},
         url={http://dx.doi.org/10.1088/0266-5611/23/4/018},
      review={\MR{2348728 (2008i:47065)}},
}

\bib{MR2105088}{book}{
      author={Simon, Barry},
       title={Orthogonal polynomials on the unit circle. {P}art 1},
      series={American Mathematical Society Colloquium Publications},
   publisher={American Mathematical Society},
     address={Providence, RI},
        date={2005},
      volume={54},
        ISBN={0-8218-3446-0},
        note={Classical theory},
      review={\MR{2105088 (2006a:42002a)}},
}

\bib{MR2105089}{book}{
      author={Simon, Barry},
       title={Orthogonal polynomials on the unit circle. {P}art 2},
      series={American Mathematical Society Colloquium Publications},
   publisher={American Mathematical Society},
     address={Providence, RI},
        date={2005},
      volume={54},
        ISBN={0-8218-3675-7},
        note={Spectral theory},
      review={\MR{2105089 (2006a:42002b)}},
}

\bib{MR2647154}{article}{
      author={Weikard, Rudi},
      author={Zinchenko, Maxim},
       title={The inverse resonance problem for {CMV} operators},
        date={2010},
        ISSN={0266-5611},
     journal={Inverse Problems},
      volume={26},
      number={5},
       pages={055012, 10},
         url={http://dx.doi.org/10.1088/0266-5611/26/5/055012},
      review={\MR{2647154 (2011j:47098)}},
}

\end{biblist}
\end{bibdiv}

\end{document}